\documentclass[12pt]{article}

\usepackage{amssymb}%
\usepackage{amsmath}%
\usepackage{amsthm}%

\usepackage{fullpage}%
\usepackage{xcolor}%

\allowdisplaybreaks%

{\theoremstyle{plain}%
 \newtheorem{theorem}{Theorem}

 \newtheorem{lemma}{Lemma}%
}
{\theoremstyle{remark}

}
{\theoremstyle{definition}

}

\begin{document}

\begin{center}
{\large On a Ramanujan-type series associated with the Heegner number 163}
\end{center}

\begin{center}
{\textsc{John M. Campbell} } 

 \ 

\end{center}

\begin{abstract}
 Using the Wolfram {\tt NumberTheory} package and the {\tt Recognize} command, together with numerical estimates involving the elliptic lambda and elliptic 
 alpha functions, Bagis and Glasser, in 2013, introduced a conjectural Ramanujan-type series related to the class number $h(-d) = 1$ for a 
 quadratic form with discriminant $d = 163$. This conjectured series is of level one and has positive terms, and recalls the Chudnovsky brothers' alternating 
 series of the same level, given the connection between the Chudnovsky--Chudnovsky formula and the Heegner number $d = 163$ such that 
 $\mathbb{Q}\left( \sqrt{-d} \right)$ has class number one. We prove Bagis and Glasser's conjecture by proving evaluations for $\lambda^{\ast}(163)$ and 
 $\alpha(163)$, which we derive using the Chudnovsky brothers' formula together with the analytic continuation of a formula due to the Borwein brothers for 
 Ramanujan-type series of level one. As a byproduct of our method, we obtain an infinite family of Ramanujan-type series for $\frac{1}{\pi}$ generalizing the 
 Chudnovsky algorithm. 
\end{abstract}



{\footnotesize {lambda function}}

\section{Introduction}
 Bagis and Glasser, in 2013 \cite{BagisGlasser2013}, obtained the following infinite family of Ramanujan-type series: 
\begin{equation}\label{BGfamily}
 \frac{3}{ \pi \sqrt{r} \sqrt{1 - J_{r}} } 
 = \sum_{n=0}^{\infty} \frac{ \left( \frac{1}{6} \right)_{n} 
 \left( \frac{5}{6} \right)_{n} \left( \frac{1}{2} \right)_{n} }{ (n!)^{3} } 
 (J_{r})^{n} (6n + 1 - T_{r}), 
\end{equation}
 where 
\begin{equation}\label{Jrdefinition}
 J_{r} := 1728 j_{r}^{-1} = 4 \beta_{r} (1 - \beta_{r}) 
\end{equation}
 and 
\begin{equation}\label{Trdefinition}
 T_{r} := \frac{ 2 j_{r}^{1/3} \sigma(r) G_{r}^{8} }{ \sqrt{r} \sqrt{j_{r} - 1728}}, 
\end{equation}
 referring to \cite{BagisGlasser2013} for details. 
 Following Bagis and Glasser's work in \cite{BagisGlasser2013}, 
 the \emph{class number} $h(-d)$ of a quadratic form of discriminant $d$ is such that
\begin{equation}\label{eqJacobi}
 h(-d) = -\frac{w(d)}{2d} \sum_{n=1}^{d-1} \left( \frac{-d}{n} \right) n, 
\end{equation}
 where $\left( \frac{n}{m} \right) $ denotes the Jacobi symbol and $w(3) = 6$, $w(4) = 4$, and $w(d) = 2$ otherwise. Since \eqref{eqJacobi} reduces to 
 $1$ for $d = 163$, it was suggested by Bagis and Glasser that this could be indicative of a greater likelihood of determining relatively simple closed forms for 
 the coefficients involved in the Ramanujan-type series in \eqref{BGfamily}. Bagis and Glasser \cite{BagisGlasser2013} applied the Wolfram Mathematica 
 {\tt NumberTheory} package and the {\tt Recognize} command to determine conjectured 
 closed forms for $J_{163}$ and $T_{163}$, based on numerical estimates of $J_{163}$ and $T_{163}$. This led Bagis and Glasser 
 \cite{BagisGlasser2013} to provide a conjectured Ramanujan-type series, given by setting the conjectured values for $J_{163}$ and $T_{163}$ into 
 \eqref{BGfamily}. This purported Ramanujan-type series would provide a natural companion to the famous Chudnovsky--Chudnovsky 
 formula \cite{ChudnovskyChudnovsky1988} 
\begin{equation}\label{Chudnovskymain}
 \frac{711822400}{\sqrt{10005}\, \pi } = 
 \sum_{n=0}^{\infty} \frac{ \left( \frac{1}{6} \right)_{n} \left( \frac{1}{2} \right)_{n} \left( \frac{5}{6} \right)_{n} }{ \left( 1 \right)_n^{3} }
 \left( -\frac{1}{53360} \right)^{3n}
 \left( \frac{13591409}{6} + 90856689 n \right). 
\end{equation}
 The Chudnovsky brothers' formula in \eqref{Chudnovskymain}
 is closely related to the same Heegner number $ d = 163$. 
 Since this Heegner number gives the value of $1$ in \eqref{eqJacobi}, 
 the quadratic field $\mathbb{Q}\left( \sqrt{-163} \right)$ has class number $1$. 

 The connection between the $r = 163$ case of \eqref{BGfamily} and the Chudnovsky brothers' formula in \eqref{Chudnovskymain} has not 
 previously been investigated, and the conjectured values for $J_{163}$ and $T_{163}$ have not previously been proved. The connection between 
 \eqref{Chudnovskymain} and the Ramanujan-type series obtained by setting $r = 163$ in \eqref{BGfamily} is of a nontrivial nature. This is evidenced by 
 the very unwieldy algebraic coefficients involved in Bagis and Glasser's conjecture, in contrast to the rational series in \eqref{Chudnovskymain}. 

 In this article, we succeed in proving Bagis and Glasser's conjectured Ramanujan-type series, by proving closed forms for $\lambda^{\ast}(163)$ and $ 
 \alpha(163)$, referring to Section \ref{sectionBG} for details. As a byproduct of the technique we introduce to prove closed forms for the coefficients 
 involved in the $r = 163$ case of \eqref{BGfamily}, we obtain an infinite family of generalizations of the famous 
 Chudnovsky--Chudnovsky formula in \eqref{Chudnovskymain}. 
 This infinite family is given explicitly in terms of the elliptic lambda and elliptic alpha functions. 

  The value of Ramanujan's class invariant $G_{163} = 2^{-1/4} x$, where $x^{3} - 6x^2 + 4x - 2 =0$, was given in 
  Part V of Ramanujan's    Notebooks \cite[p.\ 194]{Berndt1998}. 
 However, this cannot be applied in any direct way to prove Bagis and Glasser's conjecture. 
 This becomes apparent if we consider 
 the family of Ramanujan-type series of level $1$ 
 given by the Borwein brothers \cite[p.\ 365]{BorweinBorwein1988} \cite[p.\ 183]{BorweinBorwein1987} 
 and reproduced below (see also \cite{Milla2018}): 
\begin{equation}\label{publishedell1}
 \frac{1}{\pi} = \sum_{n=0}^{\infty} 
 \frac{ \left( \frac{1}{6} \right)_{n} \left( \frac{1}{2} \right)_{n} \left( \frac{5}{6} \right)_{n} 
 }{ \left( n! \right)^{3} } f_{n}(N) \left( J_{N}^{-1/2} \right)^{2n+1}, 
\end{equation}
 where 
\begin{align}
\begin{split}
 f_{n}(N) := & 
 \frac{1}{3\sqrt{3}} [ \sqrt{N} \sqrt{1 - G_{N}^{-24}} + 2(\alpha(N) - \sqrt{N} k_{N}^{2}) 
 (4 G_{N}^{24} - 1)] \\
 & + n \sqrt{N} \frac{2}{3 \sqrt{3}} [(8 G_{N}^{24} + 1)\sqrt{1 - G_{N}^{-24}}]. 
\end{split}\label{20230870747171787P7M71A}
\end{align}
 The values for $\lambda^{\ast}(163)$ and $\alpha(163)$ have not been given or used
 in the Chudnovsky brothers' derivation of \eqref{Chudnovskymain} \cite{ChudnovskyChudnovsky1988}
 or in subsequent work. 
       Bagis and Glasser's techniques    \cite{BagisGlasser2013} do not    provide any way of  
      proving any evaluation for  the required  value  $\alpha(163)$.   

\section{Bagis and Glasser's conjecture}\label{sectionBG}
 The formula in \eqref{BGfamily} given by Bagis and Glasser \cite{BagisGlasser2013} 
 can be shown to be equivalent to Equation (5.5.18) from the seminal \emph{Pi and the AGM} text 
 \cite{BorweinBorwein1987}. This equivalent formula is given in \eqref{publishedell1}. 
 We apply a closely related formula to prove Bagis and Glasser's conjecture concerning the Heegner number 163. 

 The \emph{elliptic lambda function} $\lambda^{\ast}(r)$ \cite[p.~67, eq.~(3.2.2)]{BorweinBorwein1987} 
 is such that $ \lambda^{\ast}(r):= k_r$ 
 for $0 < k_{r} < 1$, where $ \frac{ \mathbf{K}' }{\mathbf{K}}(k_{r}) = \sqrt{r}$. 
 The \emph{elliptic alpha function} \cite[p.~152, eq.~(5.1.2)]{BorweinBorwein1987} is such that 
\begin{equation}\label{mainalpha}
	\alpha(r) = \frac{\pi}{4 \mathbf{K}^{2}} - \sqrt{r} \left( \frac{\mathbf{E}}{\mathbf{K}} - 1 \right). 
\end{equation}
 The complete elliptic integrals of the first and second kinds are, respectively, such that 
\begin{equation}\label{KEdefinition}
 	\mathbf{K}(k) := \int_{0}^{\pi/2} \frac{d\theta}{\sqrt{1 - k^2 \sin^2 \theta}} \ \ \ 
 \text{and} \ \ \ 
 \mathbf{E}(k) := \int_{0}^{\pi/2} \sqrt{1 - k^2 \sin^2 \theta} \, d\theta. 
\end{equation}
 The argument $k$ in \eqref{KEdefinition} is referred to as the \emph{modulus}. 

 A \emph{Ramanujan-type series} is of the form
\begin{equation}\label{maindefinition}
	\frac{1}{\pi} = \sum_{n = 0}^{\infty} \frac{\left( \frac{1}{2} \right)_{n} \left( \frac{1}{s} \right)_{n} 
		\left(1 - \frac{1}{s} \right)_{n} }{ \left( 1 \right)_{n}^{3} } z^n (a + b n), 
\end{equation}
 where $s \in \{ 2, 3, 4, 6 \}$ and where $z$, $a$, and $b$ are real and algebraic. 
 The level of \eqref{maindefinition} is $ \ell = 4\sin^2\frac{\pi}{s}$ 
 and refers to the level of the modular form parametrizing \eqref{maindefinition}. 
 Equation (5.5.18) from \emph{Pi and the AGM} \cite{BorweinBorwein1987} 
 may be formulated as follows. 
 If we set 
\begin{equation}\label{xanalytic}
 x = 4 \left( \left(\lambda^{\ast}(r)\right)^2 - \left(\lambda^{\ast}(r)\right)^4\right), 
\end{equation}
 then the following Ramanujan-type series evaluation holds true: 
 \begin{align}\label{eq:s-6-general}
 \frac{1}{\pi} = \sum_{n = 0}^{\infty} \frac{ \left( \frac{1}{6} \right)_{n} \left( \frac{1}{2} \right)_{n} \left( \frac{5}{6} \right)_{n} }{ \left( 1 \right)_{n}^3} z^{n} (a+bn),
 \end{align}
 where 
 \begin{align}
 z &= \frac{27 x^2}{(4-x)^3}, \label{zpos} \\[1.5ex]
 a &= \frac{2 (4-x) \alpha (r)+ \left(x-4+4 \sqrt{1-x}\right) \sqrt{r}}{\sqrt{(4-x)^{3}}}, \label{20230777273157777AM2A} \\[1.5ex]
 b &= \frac{2 (x+8) \sqrt{1-x} \sqrt{r}}{\sqrt{(4-x)^{3}}}, \label{finalpositive} 
 \end{align} 
 provided that $|z|<1$. As noted by the Borwein brothers \cite[p.\ 368]{BorweinBorwein1988}, Ramanujan gave two series for $\frac{1}{\pi}$ of level $1$ 
 with a positive convergence rate, and these two series are given by special cases of \eqref{eq:s-6-general}, subject to the relations in 
 \eqref{zpos}--\eqref{finalpositive}. The formulation of \eqref{BGfamily} by the Borwein brothers in both \cite{BorweinBorwein1988} and 
 \cite{BorweinBorwein1987} only admits positive convergence rates, but it was stated in \cite[p.\ 368]{BorweinBorwein1988} that this formulation may be 
 extended, via analytic continuation, to obtain a corresponding family of Ramanujan-type series for $\ell = 1$ with negative convergence rates. An 
 explicit evaluation for the coefficients and the convergence rates for this family is given as follows, and has not appeared in 
 \cite{BorweinBorwein1988,BorweinBorwein1987} or in subsequent literature. Let $r>1$ and let $x$ be as in \eqref{xanalytic}. Then 
 the formula in \eqref{eq:s-6-general} holds for $z = -\frac{27 x}{(1-4 x)^3}$ and $b = \frac{(1+8x) \sqrt{1-x} \sqrt{r}}{\sqrt{(1-4x)^{3}}}$ and 
\begin{equation}\label{anegconv}
 a = \frac{2 (1-4x) \alpha (r)+ \left(4x-1+ \sqrt{1-x}\right) \sqrt{r}}{2\sqrt{(1-4x)^{3}}} 
\end{equation}
 provided that $|z|<1$. 

By inputting expressions as in 
\begin{verbatim}
Recognize[N[J163,1500],16,x]
\end{verbatim}
 and 
\begin{verbatim}
Recognize[N[T163,1500],16,x]
\end{verbatim}
 into Mathematica, for numerical estimates of $J_{163}$ and $T_{163}$ derived from \eqref{Jrdefinition} and \eqref{Trdefinition}, 
 this led Bagis and Glasser 
 \cite{BagisGlasser2013} to conjecture that the following evaluations hold for $J_{163}$ and $T_{163}$: $$ J_{163} = 
 \frac{4 \left( C_{1} - C_{2} \frac{1}{\sqrt[3]{-A_{1} + \sqrt{489} \text{B1}}}+30591288 \sqrt[3]{-A_{1} + \sqrt{489} 
 \text{B1}}\right)}{10792555251621895860488211571345343375}, $$ where 
\begin{align*}
 & A_{1} = 12737965652562547164590026038483234248161827096523072256574968383, \\ 
 & B_{1} = 229038073182066825378006485964950394558349727761749294205546402325349, \\
 & C_{1} = 8808429913332498766352891, \\
 & C_{2} = 902206261147132595923169636910570558029813352485594880, 
\end{align*} 
 and $$ T_{163} = \frac{5 \left(12948195754365757115+8 \sqrt[3]{A_{2} - B_{2} \sqrt{489}}+8 \sqrt[3]{A_{2} + B_{2} 
 \sqrt{489}}\right)} {83470787671093501833}, $$ where 
\begin{align*}
 & A_{2} = 3802386862487392962897493239274992371253057854289262, \\
 & B_{2} = 3865464212119923579732688315287754932290919450. 
\end{align*}
 Numerical evidence suggests that 
\begin{equation}\label{numericalsuggests}
 \frac{1}{\pi} = \sum_{n=0}^{\infty} 
 \frac{ \left( \frac{1}{6} \right)_{n} \left( \frac{1}{2} \right)_{n} 
 \left( \frac{5}{6} \right)_{n} }{ \left( n! \right)^{3} } \mathcal{J}^{n} 
 \left( \frac{\sqrt{1 - \mathcal{J}} \sqrt{163} (1 - \mathcal{T})}{3} + 2 \sqrt{1 - \mathcal{J}} \sqrt{163} n \right), 
\end{equation}
 where $\mathcal{J}$ and $\mathcal{T}$ respectively denote
 Bagis and Glasser's conjectured closed forms for $J_{163}$ and $T_{163}$ \cite{BagisGlasser2013}. 
 The purpose of the current section is to prove Bagis and Glasser's conjectured formula in 
 \eqref{numericalsuggests}. As suggested by Bagis and Glasser \cite{BagisGlasser2013}, 
 a main source of interest in this Ramanujan-type series is due to how it provides about 
 32 digits per term. 

\begin{lemma}\label{lambda163}
 The elliptic lambda function is such that $$ \lambda^{\ast}(163) = \frac{\sqrt{\frac{1}{2} \left(-1-\frac{4270934400}{\sqrt[3]{1+557403 
 \sqrt{489}}}+80040 \sqrt[3]{1+557403 \sqrt{489}}\right)}}{2 \sqrt{-1-\sqrt{1+\frac{1}{4} \left(-1-\frac{4270934400}{\sqrt[3]{1+557403 
 \sqrt{489}}}+80040 \sqrt[3]{1+557403 \sqrt{489}}\right)}}}. $$ Writing $x = 4( \left( \lambda^{\ast}(163)\right)^{2} - \left( \lambda^{\ast}(163) 
 \right)^{4} )$, the elliptic alpha function is such that $$ \alpha(163) = \frac{13591409 \sqrt{1-4 x}}{426880 \sqrt{10005}}+\frac{\sqrt{163} \sqrt{1 
 -x}}{8 x-2}+\frac{\sqrt{163}}{2}. $$ 
\end{lemma}

\begin{proof}
 We normalize the Chudnovsky--Chudnovsky formula shown in \eqref{Chudnovskymain} and proved in \cite{ChudnovskyChudnovsky1988} so that 
\begin{align}
\begin{split}
 \frac{1}{\pi} = & \sum_{n=0}^{\infty} 
 \frac{ \left( \frac{1}{6} \right)_{n} \left( \frac{1}{2} \right)_{n} 
 \left( \frac{5}{6} \right)_{n} }{ \left( n! \right)^{3} } 
 \left( -\frac{1}{53360} \right)^{3n} \times \\
 & \ \ \ \ \left( \frac{13591409}{426880 \sqrt{10005}} + 
 \frac{272570067}{213440 \sqrt{10005}} n \right). 
\end{split} \label{normalizeChudnovsky}
\end{align}
 By the specialization of Bailey's transformation \cite[p.\ 181, eq.\ (5.5.9)]{BorweinBorwein1987} 
\begin{equation}\label{Baileycubicanother}
 \sum_{n = 0}^{\infty} \frac{ \left( \frac{1}{2} \right)_{n}^{3} }{ \left( 1 \right)_{n}^{3} } 
 x^{n} = \frac{1}{\sqrt{1 - 4 x}} \sum_{n = 0}^{\infty} \frac{ \left( \frac{1}{6} \right)_{n} 
 \left( \frac{1}{2} \right)_{n} \left( \frac{5}{6} \right)_{n} }{ \left( n! \right)^{3} 
 } \left( -\frac{27 x}{(1-4x)^3} \right)^{n}, 
\end{equation}
 we may obtain an equivalent version of the 
 normalized Chudnovsky formula in the following way. 
 By applying the generating 
 function evaluation (cf.\ \cite[p.~180, Theorem 5.7(a), eq.~(i)]{BorweinBorwein1987}) 
\begin{equation}\label{gfcubedcentral}
 	\sum_{n=0}^{\infty} \frac{ \left( \frac{1}{2} \right)_{n}^{3} }{ \left( 1 \right)_{n}^{3} } x^{n} 
 	= \frac{4 \mathbf{K}^2 \left( \sqrt{\frac{1 - \sqrt{1-x}}{2}} \right)}{\pi ^2} 
\end{equation}
   and the term-by-term derivatives of the series in \eqref{gfcubedcentral}     to rewrite \eqref{Baileycubicanother} and the term-by-term derivatives of   
  \eqref{Baileycubicanother},     this gives us that the normalized Chudnovsky series  
  is equal to the right-hand side of \eqref{eq:s-6-general} 
 for the $a$-value in \eqref{anegconv}
 and for   $z = -\frac{27 x}{(1-4 x)^3}$ and $b = \frac{(1+8x) \sqrt{1-x} \sqrt{r}}{\sqrt{(1-4x)^{3}}}$, 
 and for  the unique value of $r$ such that 
\begin{equation}\label{existsfixed}
 \left( -\frac{1}{53360} \right)^{3} = -\frac{108 \left( \left( \lambda^{\ast}(r)\right)^2 - 
 \left( \lambda^{\ast}(r)\right)^4\right)}{\left(1 - 
 16 \left( \left( \lambda^{\ast}(r) \right)^2 - 
 \left( \lambda^{\ast}(r) \right)^4\right)\right)^3}, 
\end{equation}
 referring to \cite{Milla2018} for details. We then take the unique real value $x$ such that $-\frac{1}{53360^3} = -\frac{27x}{(1-4x)^3}$. By then 
 rewriting the coefficient of $n$ in the polynomial factor in \eqref{normalizeChudnovsky} as 
 $ \frac{(1+8x) \sqrt{1-x} \sqrt{r}}{\sqrt{(1-4x)^{3}}}$, 
 this gives us the desired value $r = 163$. 
 So, we obtain the desired closed form for $\lambda^{\ast}(r)$ from \eqref{existsfixed}. By then rewriting the constant term in the polynomial 
 factor in \eqref{normalizeChudnovsky} as in \eqref{anegconv}, we obtain the desired closed form for $\alpha(163)$. 
\end{proof}

\begin{theorem}\label{maintheorem}
 Bagis and Glasser's Ramanujan-type series evaluation in \eqref{numericalsuggests} holds true, for the conjectured values of $J_{163}$ and $T_{163}$. 
\end{theorem}

\begin{proof}
 This follows in a direct way by setting the values for $\lambda^{\ast}(163)$ and $\alpha(163)$ given  in Lemma \ref{lambda163} into \eqref{xanalytic}, for 
 the values of $z$, $a$, and $b$  indicated in \eqref{zpos}--\eqref{finalpositive}. 
\end{proof}

   The above Theorem may be regarded as providing a series acceleration,    since Bagis and Glasser's formula gives about 32 digits per term, compared to  
  the Chudnovsky brothers' algorithm giving about 14 digits per term.   

\section{Conclusion}\label{sectionConclusion}
 Our infinite family of Ramanujan-type series involving \eqref{anegconv} has led us to discover a new  Ramanujan-type series with a cubic convergence 
 rate that was not provided in \cite{BorweinBorwein1993}. This new series corresponds to the $r = 243$ 
 case of \eqref{anegconv}, noting that $243 = 3^{5}$, so that this non-squarefree case does not agree with the class number 
 three cases covered in \cite{BorweinBorwein1993}, and noting that we can use known recursions for Ramanujan's class invariants 
 \cite[p.\ 145]{BorweinBorwein1987} and for the elliptic alpha function \cite[p.\ 160]{BorweinBorwein1987} to prove the following result, given by a 
 Ramanujan-type series as in \eqref{eq:s-6-general} for the following values of $z$, $a$, and $b$. 
\begin{verbatim}
z = Root[6561 + 7046099782711303104000*#1 - 8241190499340288000000*#1^2 + 
4245232549888000000000*#1^3 & , 1, 0]
a = -Root[-529386137841972399112323 + 5193712075415612200704000*#1^2 + 
279610488520114176000000*#1^4 + 271694883192832000000000*#1^6 & , 1, 0]
b = -Root[-43766201502243114733034506827 + 
194827240993239940764096000*#1^2 - 1092165237528662016000000*#1^4 + 
4245232549888000000000*#1^6 & , 1, 0]
\end{verbatim}
 This gives about 18 digits per term. We encourage the systematic exploration of 
 Ramanujan-type series derived by analogy with our proof of Theorem \ref{maintheorem}, 
 to yield faster versions of the Chudnovsky algorithm. 
 
\subsection*{Acknowledgements}
 The author is grateful to acknowledge support from a Killam Postdoctoral Fellowship. The author is sincerely thankful to Shane Chern for many useful 
 discussions. The Reviewer feedback that has been provided has been very helpful and has led to considerable improvements of 
 this article. 

{\footnotesize

 \ 

Department of Mathematics and Statistics

Dalhousie University

{\tt jmaxwellcampbell@gmail.com}

}


\begin{thebibliography}{99}

\bibitem{BagisGlasser2013}
 Bagis, N. D., Glasser, M. L. (2013). 
 Ramanujan type {$1/\pi$} approximation formulas. 
 \emph{J. Number Theory} 
 133(10): 3453--3469. 

\bibitem{Berndt1998}
 Berndt, B. C.\ (1998). 
 \emph{Ramanujan's Notebooks. {Part} {V}}, 
 New York, NY: Springer. 

\bibitem{BorweinBorwein1993}
 Borwein, J. M., Borwein, P. B. (1993). 
 Class number three {R}amanujan type series for {$1/\pi$}. 
 \emph{J. Comput. Appl. Math.} 
 46(1-2): 281--290. 

\bibitem{BorweinBorwein1988}
 Borwein, J. M., Borwein, P. B. (1988). 
 More {R}amanujan-type series for {$1/\pi$}. 
 In: \emph{Ramanujan revisited ({U}rbana-{C}hampaign, {I}ll., 1987)}. 
 Boston, MA: Academic Press. 

\bibitem{BorweinBorwein1987}
 Borwein, J. M., Borwein, P. B.\ (1987). 
 \emph{Pi and the {AGM}}, 
 Canadian Mathematical Society Series of Monographs and Advanced Texts. 
 New York: John Wiley \& Sons, Inc. 

\bibitem{ChudnovskyChudnovsky1988}
 Chudnovsky, D. V., Chudnovsky, G. V.\ (1988). 
 Approximations and complex multiplication according to {R}amanujan. 
 In: \emph{Ramanujan revisited ({U}rbana-{C}hampaign, {I}ll., 1987)}. 
 Boston, MA: Academic Press. 

\bibitem{Milla2018}
 Milla, L.\ (2018). 
 A detailed proof of the Chudnovsky formula with means of basic complex analysis--Ein ausf\"{u}hrlicher Beweis der 
 Chudnovsky-Formel mit elementarer Funktionentheorie. 
 arXiv:1809.00533. 

\end{thebibliography}
\end{document}